\documentclass[11pt, final]{article}
\usepackage{a4}
\usepackage{amsmath}%
\usepackage{amstext}%
\usepackage{amssymb}%
\usepackage{showkeys}%
\usepackage{epsfig}%
\usepackage{cite}
\usepackage{tikz}
\usepackage{subfig}
\usepackage{caption}
\usepackage{multirow}
\usepackage{booktabs}
\usepackage{floatrow}

\usepackage{listings}
\usepackage{geometry}
\newtheorem{theorem}{Theorem}

\newtheorem{lemma}[theorem]{Lemma}

\newtheorem{rem}[theorem]{Remark}
\newenvironment{remark}{\rem\rm}{\endrem}

\newcounter{unnumber}

\newtheorem{prf}[unnumber]{Proof.}
\newenvironment{proof}{\prf\rm}{\hfill{$\blacksquare$}\endprf}
\newcommand{\R}{\mathbb{R}}%
\newcommand{\N}{\mathbb{N}}%
\DeclareMathOperator*\ran{ran}%
\DeclareMathOperator*\argmin{argmin}

\title{A second order dynamical system with Hessian-driven damping and penalty term associated to variational inequalities}

\author{Radu Ioan Bo\c{t} \thanks{University of Vienna, Faculty of Mathematics, Oskar-Morgenstern-Platz 1, A-1090 Vienna, Austria,
email: radu.bot@univie.ac.at.} \and
Ern\"{o} Robert Csetnek \thanks {University of Vienna, Faculty of Mathematics, Oskar-Morgenstern-Platz 1, A-1090 Vienna, Austria,
email: ernoe.robert.csetnek@univie.ac.at. Research supported by FWF (Austrian Science Fund), Lise Meitner Programme, project M 1682-N25.}}

\begin{document}
\maketitle

\noindent \textbf{Abstract.} We consider the minimization of a convex objective function subject to the set of minima of another convex function, under the assumption that both functions  are twice continuously
differentiable. We approach this optimization problem from a continuous perspective by means of a second order dynamical system with Hessian-driven damping and a penalty term corresponding to the constrained function.  By constructing appropriate energy functionals, we  prove weak convergence of the trajectories generated by this differential equation to a minimizer of the optimization problem
as well as convergence for the objective function values along the trajectories. The performed investigations rely on Lyapunov analysis in combination with the continuous version of the Opial Lemma. In case the objective function is strongly convex, we can even show strong convergence of the trajectories.
\vspace{1ex}

\vspace{1ex}


\section{Introduction}\label{sec1}

The Newton-like dynamical system
\begin{equation}\label{dyn-syst-psi0-intr}\left\{
\begin{array}{ll}
\ddot x(t) + \gamma\dot x(t) + \lambda\nabla^2\Phi(x(t))(\dot x(t)) +  \nabla \Phi(x(t))=0\\
x(0)=u_0, \dot x(0)=v_0,
\end{array}\right.\end{equation}
has been investigated by Alvarez, Attouch, Bolte and Redont in \cite{alv-att-bolte-red} in connection with the
optimization problem
\begin{equation}\label{opt-intr-phi}\inf_{x\in {\cal H}}\Phi(x).\end{equation}
Here, ${\cal H}$ is a real Hilbert space
endowed with inner product $\langle\cdot,\cdot\rangle$ and associated norm $\|\cdot\|\!=\!\sqrt{\langle \cdot,\cdot\rangle}$,
$u_0,v_0\in {\cal H}$ are the initial data, $\lambda,\gamma>0$, while $\nabla \Phi$ and $\nabla^2\Phi$ denote the gradient and Hessian
of the function $\Phi:{\cal H}\rightarrow\R$, respectively. We speak here about a second order system in time (through the presence of the acceleration term $\ddot x(t)$, which is associated to inertial effects) and in space (through $\nabla^2\Phi(x(t))$). One can also notice the presence of the geometric damping that acts on the velocity through the Hessian of the function $\Phi$.

As underlined in \cite{alv-att-bolte-red}, the dynamical system \eqref{dyn-syst-psi0-intr} can be seen as a mixture of the continuous Newton method
\begin{equation}\label{cn}\nabla^2\Phi(x(t))(\dot x(t))+\nabla\Phi(x(t))=0,\end{equation}
investigated by Alvarez and P\'{e}rez in \cite{alv-per1998}, with the heavy ball with friction system
\begin{equation}\label{heavy-ball}\ddot x(t) + \gamma\dot x(t)  +  \nabla \Phi(x(t))=0,\end{equation}
studied for the first time in Polyak \cite{polyak} and Antipin \cite{antipin}. Due to this remarkable fact, the dynamical system
\eqref{dyn-syst-psi0-intr} possesses most of the advantages of the systems \eqref{cn} and \eqref{heavy-ball}. We refer the reader to
\cite{alv-att-bolte-red, alv-per1998, att-bolte-red2002, att-mainge-red2012, att-peyp-red2016, att-red2001} for more insights on Newton-type dynamics and their motivations coming from mechanics and control theory.

The aim of this paper is to associate a second order Newton-type
dynamical system to the optimization problem
\begin{equation}\label{opt-intr}\inf_{x\in\argmin\Psi}\Phi(x),\end{equation}
where $\Phi,\Psi: {\cal H}\rightarrow \R$ are convex and twice differentiable functions, and to investigate its asymptotic properties.

Let us notice that, due to the first order optimality conditions, solving \eqref{opt-intr} can be formulated as
a variational inequality of the form
\begin{equation}\label{vi}
\mbox{find} \ x \in {\argmin\Psi}  \ \mbox{such that} \ \langle \nabla\Phi(x), y-x \rangle \geq 0 \ \forall y \in \argmin\Psi,
\end{equation}
where $\argmin\Psi$ denotes the set of minimizers of  $\Psi$ over ${\cal H}$.

Attouch and Czarnecki have assigned  in \cite{att-cza-10} to \eqref{opt-intr} the nonautonomous first order dynamical system
\begin{equation}\label{1ord-att-cz}\dot x(t)+\nabla \Phi(x(t))+\beta(t)\nabla \Psi(x(t))=0,\end{equation}
where $\beta:[0,+\infty)\rightarrow(0,+\infty)$ is a function of time assumed to tend to $+\infty$ as $t\rightarrow +\infty$, which penalizes the constrained function.
Several convergence results of the trajectories generated by \eqref{1ord-att-cz} to the solution set of \eqref{opt-intr} have been reported in \cite{att-cza-10} under the key assumption
\begin{equation}\label{ass-att-cz}\forall p\in\ran N_{\argmin\Psi} \ \int_0^{+\infty}
\beta(t)\left[\Psi^*\left(\frac{p}{\beta(t)}\right)-\sigma_{\argmin\Psi}\left(\frac{p}{\beta(t)}\right)\right]dt<+\infty,
\end{equation}
where $\Psi^*: {\cal H}\rightarrow \R\cup\{+\infty\}$ is the Fenchel-Legendre transformation of $\Psi$:
$$\Psi^*(p)=\sup_{x\in{\cal H}}\{\langle p,x\rangle-\Psi(x)\} \ \forall p\in {\cal H};$$
$\sigma_{\argmin\Psi}: {\cal H}\rightarrow \R\cup\{+\infty\}$ is the support function of the set $\argmin\Psi$:
$$\sigma_{\argmin\Psi}(p)=\sup_{x\in {\argmin\Psi}}\langle p,x\rangle \ \forall p\in {\cal H};$$
and $N_{\argmin\Psi}$ is the normal cone to the set $\argmin\Psi$, defined by
$$N_{\argmin\Psi}(x)=\{p\in{\cal H}:\langle p,y-x\rangle\leq 0 \ \forall y\in \argmin\Psi\}$$ for $x \in \argmin\Psi$ and
$N_{\argmin\Psi}(x)=\emptyset$ for $x\not\in \argmin\Psi$. Finally, $\ran N_{\argmin\Psi}$ denotes the range of the normal cone $N_{\argmin\Psi}$, that is, $p\in\ran N_{\argmin\Psi}$ if and only if there exists $x\in \argmin\Psi$ such that $p\in N_{\argmin\Psi}(x)$.
Let us notice that for $x\in {\argmin\Psi}$ one has $p\in N_{\argmin\Psi}(x)$ if and only if $\sigma_{\argmin\Psi}(p)=\langle p,x\rangle$.

We present a situation where the above condition \eqref{ass-att-cz} is fulfilled. According to \cite{att-cza-10}, if we take
$$\Psi(x)=\frac{1}{2}\inf_{y\in C}\|x-y\|^2,$$ for a nonempty, convex and closed set $C \subseteq {\cal H}$, then the condition
\eqref{ass-att-cz} is fulfilled if and only if $$\int_0^{+\infty}\frac{1}{\beta(t)} dt<+\infty,$$ which is trivially satisfied for
$\beta(t)=(1+t)^\alpha$ with $\alpha>1$.

The paper of Attouch and Czarnecki \cite{att-cza-10} was the starting point of a considerable number of research articles devoted to this subject,
including those addressing generalizations to variational inequalities formulated with maximal monotone operators
(see \cite{att-cza-10, att-cza-peyp-c, att-cza-peyp-p, noun-peyp, peyp-12, b-c-penalty-svva, b-c-penalty-vjm, banert-bot-pen,
b-c-dyn-pen, att-cab-cz, att-maing, b-c-dyn-sec-ord-pen, b-c-levenberg}). We refer also to the above-listed references for more general formulations
of the key assumption \eqref{ass-att-cz} and for further examples for which these conditions are satisfied.

In  \cite{b-c-dyn-sec-ord-pen} we approached  the optimization problem \eqref{opt-intr} through
the second order nonautonomous dynamical system
\begin{equation}\label{b-c-aa}\ddot x(t) + \gamma\dot x(t) + \nabla \Phi(x(t))+\beta(t)\nabla \Psi(x(t))=0,\end{equation}
where $\gamma > 0$ and $\beta:[0,+\infty)\rightarrow(0,+\infty)$ is a function of time. Under the assumption that $\beta$ tends to $+\infty$ as $t\rightarrow +\infty$, we proved weak convergence of the generated trajectories to a minimizer of \eqref{opt-intr} as well as convergence for the objective function values along the trajectories. We refer to \cite{att-cza-16} for another variant of this system, where the objective function is penalized instead of the constraint function.

The aim of this paper is to combine Newton-like dynamics with systems of the form \eqref{b-c-aa} in order to approach from a continuous perspective the solving of the optimization problem \eqref{opt-intr}. To this end, we propose ourselves to investigate in this paper the asymptotic behavior of the dynamical system
\begin{equation}\label{dyn-syst-intr}
\ddot x(t) + \gamma\dot x(t) + \lambda\nabla^2\Phi(x(t))(\dot x(t)) + \lambda\beta(t)\nabla^2\Psi(x(t))(\dot x(t)) +
\nabla \Phi(x(t))+(\beta(t)+\lambda\dot\beta(t))\nabla \Psi(x(t))=0.\end{equation}
Condition \eqref{ass-att-cz} will be again crucial in the analysis performed. By using Lyapunov analysis in combination with the continuous version of the Opial Lemma, we prove weak convergence of the trajectories to a minimizer of the optimization problem \eqref{opt-intr}
as well as convergence for the objective function values along the trajectories. In case the objective function is strongly convex, we can even show strong convergence of the trajectories.

\section{Preliminaries}\label{sec2}

In this section we will introduce preliminary notions and results that will be useful throughout the paper.

The following statement can be interpreted as the continuous counterpart of the convergence result of quasi-Fej\'er monotone sequences. For its proofs we refer the reader to \cite[Lemma 5.1]{abbas-att-sv}.

\begin{lemma}\label{fejer-cont1} Suppose that $F:[0,+\infty)\rightarrow\R$ is locally absolutely continuous and bounded from below and that
there exists $G\in L^1([0,+\infty))$ such that for almost every $t \in [0,+\infty)$ $$\dot F(t)\leq G(t).$$
Then there exists $\lim_{t\rightarrow +\infty} F(t)\in\R$.
\end{lemma}

We will focus our investigations on the following second order dynamical system
\begin{equation}\label{dyn-syst}\left\{
\begin{array}{ll}
\!\!\!\ddot x(t) + \gamma\dot x(t) + \!\lambda\nabla^2\Phi(x(t))(\dot x(t)) + \!\lambda\beta(t)\nabla^2\Psi(x(t))(\dot x(t)) +
\nabla \Phi(x(t))+(\beta(t)+\lambda\dot\beta(t))\nabla \Psi(x(t))=0\\
\!\!\!x(0)=u_0, \dot x(0)=v_0,
\end{array}\right.\end{equation}
where $\gamma,\lambda>0$, $u_0,v_0\in {\cal H}$, provided that the following assumptions are satisfied:

\begin{align*}
(H_\Psi)& \ \Psi:{\cal H}\rightarrow [0,+\infty) \mbox{ is convex, twice differentiable
  such that } \nabla\Psi \mbox{ and }\nabla^2\Psi \\&\mbox{ are Lipschitz continuous} \mbox{ and }\argmin\Psi=\Psi^{-1}(0)\neq\emptyset;\\
(H_\Phi)& \ \Phi:{\cal H}\rightarrow \R \mbox{ is convex, twice differentiable, bounded from below, such that }\nabla\Phi\mbox{ and }
\nabla^2\Phi\\&  \mbox{ are Lipschitz continuous}
 \mbox{ and }S:=\{z\in \argmin\Psi: \Phi(z)\leq \Phi(x) \  \forall x\in\argmin\Psi\}\neq\emptyset; \\
(H_{\beta})& \ \beta :[0,+\infty) \rightarrow (0,+\infty) \mbox{ is a }C^1\mbox{-function with } \lim_{t\rightarrow+\infty}\beta(t)=+\infty
\mbox{ and it satisfies the growth}\\&\mbox{ condition }
0\leq\dot\beta\leq k\beta,\mbox{ where } 0 < k<\frac{\theta}{1+\lambda\gamma}\min\left\{\frac{2\gamma}{3}, \frac{2}{\lambda}\right\}
\mbox{ for }\theta\in(0,1).\end{align*}

We look at strong global solutions $x:[0,+\infty)\rightarrow{\cal H}$ of the dynamical system of \eqref{dyn-syst}, that is, $x$ and $\dot x$ are locally absolutely continuous (in other words, absolutely continuous on each interval $[0,b]$ for $0<b<+\infty$),
$x(0)=u_0, \dot x(0)=v_0$ and
$$\ddot x(t) + \gamma\dot x(t) + \lambda\nabla^2\Phi(x(t))(\dot x(t)) + \lambda\beta(t)\nabla^2\Psi(x(t))(\dot x(t)) +
\nabla \Phi(x(t))+(\beta(t)+\lambda\dot\beta(t))\nabla \Psi(x(t))=0$$
for almost every $t \in [0,+\infty)$.

In view of the Lipschitz continuity of $\nabla \Psi, \nabla ^2\Psi$ and $\nabla \Phi, \nabla ^2\Phi$, assumed in
$(H_\Psi)$ and $(H_\Phi)$, respectively, the existence and uniqueness of strong global solutions of \eqref{dyn-syst} is a consequence
of the Cauchy-Lipschitz-Picard Theorem (see for example \cite{att-maing, b-c-dyn-sec-ord, alv-att-bolte-red, cabot-engler-gadat-tams}).

\begin{remark}\label{rem-heavy-ball} \begin{enumerate}\item[(a)] In case $\Psi=0$, the dynamical system \eqref{dyn-syst} becomes
\begin{equation}\label{dyn-syst-psi0}\left\{
\begin{array}{ll}
\ddot x(t) + \gamma\dot x(t) + \lambda\nabla^2\Phi(x(t))(\dot x(t)) +  \nabla \Phi(x(t))=0\\
x(0)=u_0, \dot x(0)=v_0,
\end{array}\right.\end{equation}
the convergence of which has been investigated in \cite{alv-att-bolte-red} in connection with the minimization of the function
$\Phi$ over ${\cal H}$.

\item[(b)] The time discretization of second order dynamical systems leads to iterative algorithms involving inertial terms,
which basically means that every new iterate is constructed in terms of the previous two iterates (see for example
\cite{alvarez2004, alvarez-attouch2001}). In view of this observation, it makes sense to investigate a time discretized version of \eqref{dyn-syst} and
to study the convergence properties of the generated iterates in relation with the solving of the optimization problem \eqref{opt-intr}. We leave this topic as future research work.
\end{enumerate}
\end{remark}

\section{Convergence of the trajectories and of the objective function values}\label{sec3}

This section is devoted to the asymptotic analysis of the trajectory generated by the dynamical system \eqref{dyn-syst}. We show weak convergence of the trajectory $x(\cdot)$ to an optimal solution
of \eqref{opt-intr} as well as convergence for the objective function values along the trajectory as $t \rightarrow +\infty$, under the following assumption:
\begin{enumerate}
 \item[$(H)$] $\forall p\in\ran N_{\argmin \Psi} \ \int_0^{+\infty}
\beta(t)\left[\Psi^*\left(\frac{p}{\beta(t)}\right)-\sigma_{\argmin \Psi}\left(\frac{p}{\beta(t)}\right)\right]dt<+\infty$.
\end{enumerate}

\begin{remark}
\begin{itemize}
\item[(a)] Let $\delta_{\argmin\Psi}:{\cal H}\rightarrow\R\cup\{+\infty\}$ be the indicator function
of ${\argmin\Psi}$, which is the function that takes the value $0$ on the set ${\argmin\Psi}$ and $+\infty$, otherwise. Due to $\Psi\leq\delta_{\argmin \Psi}$ (see $(H_\Psi)$), we
have $$\Psi^*\geq\delta_{\argmin \Psi}^*=\sigma_{\argmin \Psi}.$$
\item[(b)] Considering the case when $\Psi=0$ (see Remark \ref{rem-heavy-ball}(a)), we have $N_{\argmin \Psi}(x)=\{0\}$ for every $x\in \argmin \Psi={\cal H}$,
$\Psi^*=\sigma_{\argmin \Psi}=\delta_{\{0\}}$ and $(H)$ trivially holds.
\end{itemize}
\end{remark}

For $\delta>0$, we considere the following energy functional that will play an important role in the analysis below:
\begin{equation}\label{e-delta}
E_{\delta}(t)=\delta\big(\Phi(x(t))+\beta(t)\Psi(x(t))\big)+\frac{1}{2}\left\|\dot x(t)+\lambda\nabla\Phi(x(t))+\lambda\beta(t)\nabla\Psi(x(t))\right\|^2 \ \forall t >0.
\end{equation}
For its derivative we have for almost every $t \in [0,+\infty)$
\begin{align*}
\dot E_{\delta}(t) = & \ \delta\langle \nabla\Phi(x(t)),\dot x(t)\rangle+\delta\beta(t)\langle\nabla\Psi(x(t)),\dot x(t)\rangle +\delta\dot\beta(t)\Psi(x(t))+\\
\langle \ddot x(t)+  \!\lambda&\nabla^2\Phi(x(t))(\dot x(t))\!+\!\lambda\beta(t)\nabla^2\Psi(x(t))(\dot x(t))+\!\lambda\dot\beta(t)\nabla\Psi(x(t)),\dot x(t)+\!\lambda\nabla\Phi(x(t))+\!\lambda\beta(t)\nabla\Psi(x(t))\rangle\\
= & \ \delta\langle \nabla\Phi(x(t)),\dot x(t)\rangle+\delta\beta(t)\langle\nabla\Psi(x(t)),\dot x(t)\rangle
+\delta\dot\beta(t)\Psi(x(t))+\\
& \langle-\gamma\dot x(t)-\nabla\Phi(x(t))-\beta(t)\nabla\Psi(x(t)),\dot x(t)+\lambda\nabla\Phi(x(t))+\lambda\beta(t)\nabla\Psi(x(t))\rangle\\
=& \ \delta\langle \nabla\Phi(x(t)),\dot x(t)\rangle+\delta\beta(t)\langle\nabla\Psi(x(t)),\dot x(t)\rangle
+\delta\dot\beta(t)\Psi(x(t)) -\gamma\|\dot x(t)\|^2-\\
&(1+\gamma\lambda)\langle \nabla\Phi(x(t)),\dot x(t)\rangle-(1+\gamma\lambda)\beta(t)\langle \nabla\Psi(x(t)),\dot x(t)\rangle-\lambda\|\nabla\Phi(x(t))\|^2-\\
& \lambda\beta^2(t)\|\nabla\Psi(x(t))\|^2-2\lambda\beta(t)\langle \nabla\Phi(x(t)),\nabla\Psi(x(t))\rangle\\
=& \ (\delta-\gamma\lambda-1)\langle \nabla\Phi(x(t)),\dot x(t)\rangle+\beta(t)(\delta-\gamma\lambda-1)\langle \nabla\Psi(x(t)),\dot x(t)\rangle
+\delta\dot\beta(t)\Psi(x(t))- \\
&\gamma\|\dot x(t)\|^2-\lambda\|\nabla\Phi(x(t))\|^2-\lambda\beta^2(t)\|\nabla\Psi(x(t))\|^2-2\lambda\beta(t)\langle \nabla\Phi(x(t)),\nabla\Psi(x(t))\rangle.
\end{align*}

Finally we obtain for almost every $t \in [0,+\infty)$
\begin{align}\label{dot-e-delta}\dot E_{\delta}(t) = & \ (\delta-\gamma\lambda-1)\langle \nabla\Phi(x(t)),\dot x(t)\rangle+\beta(t)(\delta-\gamma\lambda-1)\langle \nabla\Psi(x(t)),\dot x(t)\rangle
+\delta\dot\beta(t)\Psi(x(t))\nonumber\\
& -\gamma\|\dot x(t)\|^2-\lambda\|\nabla\Phi(x(t))+\beta(t)\nabla\Psi(x(t))\|^2.
\end{align}

Further, for $z\in S$ and \begin{equation}\label{eps}\varepsilon:=\theta\min\left\{\frac{2\gamma}{3}, \frac{2}{\lambda}\right\}\end{equation}
we consider the functional
\begin{equation}\label{e-eps}E(t)=\frac{1}{\varepsilon} E_{1+\gamma\lambda}(t)+\frac{\gamma}{2}\|x(t)-z\|^2
+\langle \dot x(t)+\lambda\nabla\Phi(x(t))+\lambda\beta(t)\nabla\Psi(x(t)),x(t)-z\rangle \ \forall t >0.\end{equation}
By using \eqref{dot-e-delta} we easily derive  for almost every $t \in [0,+\infty)$
\begin{align}\label{dot-e-eps}\dot E(t) = & \ \frac{1+\gamma\lambda}{\varepsilon}\dot\beta(t)\Psi(x(t))-\frac{\gamma}{\varepsilon}\|\dot x(t)\|^2
-\frac{\lambda}{\varepsilon}\|\nabla\Phi(x(t))+\beta(t)\nabla\Psi(x(t))\|^2 \nonumber\\
& +\gamma\langle \dot x(t),x(t)-z\rangle+\langle-\gamma\dot x(t)-\nabla\Phi(x(t))-\beta(t)\nabla\Psi(x(t)),x(t)-z\rangle\nonumber\\
& +\langle\dot x(t)+\lambda\nabla\Phi(x(t))+\lambda\beta(t)\nabla\Psi(x(t)),\dot x(t)\rangle\nonumber\\
= & \ \frac{1+\gamma\lambda}{\varepsilon}\dot\beta(t)\Psi(x(t))-\langle\nabla\Phi(x(t))+\beta(t)\nabla\Psi(x(t)), x(t)-z\rangle\nonumber\\
& -\left(\frac{\gamma}{\varepsilon}-1\right)\left\|\dot x(t)\right\|^2
 + \lambda\langle\nabla\Phi(x(t))+\beta(t)\nabla\Psi(x(t)),\dot x(t)\rangle\nonumber\\
& -\frac{\lambda}{\varepsilon}\|\nabla\Phi(x(t))+\beta(t)\nabla\Psi(x(t))\|^2\nonumber\\
\leq & \ \frac{1+\gamma\lambda}{\varepsilon}\dot\beta(t)\Psi(x(t))-\langle\nabla\Phi(x(t))+\beta(t)\nabla\Psi(x(t)), x(t)-z\rangle\nonumber\\
& -\left(\frac{\gamma}{\varepsilon}-\frac{3}{2}\right)\left\|\dot x(t)\right\|^2
-\frac{\lambda}{\varepsilon}\left(1-\frac{\varepsilon\lambda}{2}\right)\|\nabla\Phi(x(t))+\beta(t)\nabla\Psi(x(t))\|^2.
\end{align}

The following lemma will play an essential role in the asymptotic analysis of the trajectories.

\begin{lemma}\label{l1-op} Assume that $(H_\Psi)$, $(H_\Phi)$,
$(H_{\beta})$ and $(H)$ hold and let $x:[0,+\infty)\rightarrow{\cal H}$ be the trajectory generated by the dynamical system \eqref{dyn-syst}.
Then for every $z\in S$ the following statements are true:
\begin{enumerate}
 \item[(i)] $\int_0^{+\infty}\beta(t)\Psi(x(t))dt<+\infty$;
 \item[(ii)] $\exists\lim_{t\rightarrow+\infty}\int_0^t\langle \nabla\Phi(z),x(s)-z\rangle ds\in\R$;
 \item[(iii)] $\exists\lim_{t\rightarrow+\infty}\int_0^t\left(\Phi(x(s))-\Phi(z)+\left(1-\frac{k(1+\gamma\lambda)}{\varepsilon}\right)\beta(s)\Psi(x(s))\right)ds\in\R$;
 \item[(iv)] $\dot x$, $\nabla\Phi (x)+\beta\nabla\Psi (x)\in L^2([0,+\infty);{\cal H})$.
 \end{enumerate}
\end{lemma}

\begin{proof} Take an arbitrary $z\in S$. Relying on the convexity of the functions $\Phi$ and $\Psi$, the fact that $z\in \argmin\Psi$ (hence $\Psi(z)=0$)
and the non-negativity of $\beta$ and $\Psi$ we obtain  for every $t \in [0,+\infty)$
\begin{equation}\label{conv}-\langle \nabla \Phi(x(t))+\beta(t)\nabla\Psi(x(t)),x(t)-z\rangle\leq \Phi(z)-\Phi(x(t))-\beta(t)\Psi(x(t)).\end{equation}
From here and \eqref{dot-e-eps} we derive
\begin{align}\label{h-e-1} \dot E(t)\leq & \ \frac{1+\gamma\lambda}{\varepsilon}\dot\beta(t)\Psi(x(t))
+\Phi(z)-\Phi(x(t))-\beta(t)\Psi(x(t))\nonumber\\
&-\left(\frac{\gamma}{\varepsilon}-\frac{3}{2}\right)\left\|\dot x(t)\right\|^2
-\frac{\lambda}{\varepsilon}\left(1-\frac{\varepsilon\lambda}{2}\right)\|\nabla\Phi(x(t))+\beta(t)\nabla\Psi(x(t))\|^2
\end{align}
for almost every $t \in [0,+\infty)$.

By using the growth condition on $\beta$ we get
\begin{align}\label{h-e-1'} \dot E(t)\leq &
 \ \left(k\frac{1+\gamma\lambda}{\varepsilon}-1\right)\beta(t)\Psi(x(t))+\Phi(z)-\Phi(x(t)) \nonumber\\
& -\left(\frac{\gamma}{\varepsilon}-\frac{3}{2}\right)\left\|\dot x(t)\right\|^2
-\frac{\lambda}{\varepsilon}\left(1-\frac{\varepsilon\lambda}{2}\right)\|\nabla\Phi(x(t))+\beta(t)\nabla\Psi(x(t))\|^2\nonumber\\
= & -\widetilde\beta(t)\Psi(x(t))+\Phi(z)-\Phi(x(t))\nonumber\\
& -\left(\frac{\gamma}{\varepsilon}-\frac{3}{2}\right)\left\|\dot x(t)\right\|^2
-\frac{\lambda}{\varepsilon}\left(1-\frac{\varepsilon\lambda}{2}\right)\|\nabla\Phi(x(t))+\beta(t)\nabla\Psi(x(t))\|^2,
\end{align}
where \begin{equation}\label{def-beta-tilde}\widetilde\beta(t):=\left(1-k\frac{1+\gamma\lambda}{\varepsilon}\right)\beta(t)>0,\end{equation}
due to $(H_{\beta})$ and \eqref{eps}.

Furthermore, \begin{equation}\label{phi-conv}\Phi(z)-\Phi(x(t))\leq \langle -\nabla \Phi(z),x(t)-z\rangle ,\end{equation}
thus
\begin{align}\label{h-e-1''} \dot E(t)\leq &
  -\widetilde\beta(t)\Psi(x(t))-\langle \nabla \Phi(z),x(t)-z\rangle\nonumber\\
& -\left(\frac{\gamma}{\varepsilon}-\frac{3}{2}\right)\left\|\dot x(t)\right\|^2
-\frac{\lambda}{\varepsilon}\left(1-\frac{\varepsilon\lambda}{2}\right)\|\nabla\Phi(x(t))+\beta(t)\nabla\Psi(x(t))\|^2
\end{align}
for almost every $t \in [0,+\infty)$. Since $z$ is an optimal solution of \eqref{opt-intr}, the first order optimality condition
delivers \begin{equation}\label{zero-in-partial}0\in\partial (\Phi+\delta_{\argmin \Psi})(z)=\nabla\Phi(z)+ N_{\argmin \Psi}(z),\end{equation}
hence \begin{equation}\label {opt-cond} -\nabla \Phi(z)\in N_{\argmin \Psi}(z) \subseteq \ran N_{\argmin \Psi}.\end{equation}
From here and by using the Young-Fenchel inequality we obtain for every $t \in [0,+\infty)$
\begin{align}\label{ineq-conj}-\widetilde\beta(t)\Psi(x(t))-\langle \nabla \Phi(z),x(t)-z\rangle & =
\widetilde\beta(t)\left(-\Psi(x(t))-\left\langle \frac{\nabla \Phi(z)}{\widetilde\beta(t)},x(t)-z\right\rangle\right)\nonumber\\
& = \widetilde\beta(t)\left(-\Psi(x(t))+\left\langle \frac{-\nabla \Phi(z)}{\widetilde\beta(t)},x(t)\right\rangle-\sigma_{\argmin \Psi}\left(\frac{-\nabla \Phi(z)}{\widetilde\beta(t)}\right)\right)\nonumber\\
& \leq \widetilde\beta(t)\left(\Psi^*\left(\frac{-\nabla \Phi(z)}{\widetilde\beta(t)}\right)-\sigma_{\argmin \Psi}\left(\frac{-\nabla \Phi(z)}{\widetilde\beta(t)}\right)\right).\end{align}

Thus, from \eqref{h-e-1''} and \eqref{ineq-conj} we obtain for almost every $t \in [0,+\infty)$
\begin{align}\label{ineq3}\dot E(t)\leq & \
\widetilde\beta(t)\left(\Psi^*\left(\frac{-\nabla\Phi(z)}{\widetilde\beta(t)}\right)-
\sigma_{\argmin \Psi}\left(\frac{-\nabla\Phi(z)}{\widetilde\beta(t)}\right)\right)\nonumber\\
& -\left(\frac{\gamma}{\varepsilon}-\frac{3}{2}\right)\left\|\dot x(t)\right\|^2
-\frac{\lambda}{\varepsilon}\left(1-\frac{\varepsilon\lambda}{2}\right)\|\nabla\Phi(x(t))+\beta(t)\nabla\Psi(x(t))\|^2 \\
\leq & \ \widetilde\beta(t)\left(\Psi^*\left(\frac{-\nabla\Phi(z)}{\widetilde\beta(t)}\right)-
\sigma_{\argmin \Psi}\left(\frac{-\nabla\Phi(z)}{\widetilde\beta(t)}\right)\right)\label{l1},
\end{align}
where the last inequality follows from \eqref{eps} and the fact that $\theta\in(0,1)$.

By integrating the last inequality from $0$ to $T$ ($T>0$) and by taking into account $(H)$, \eqref{e-eps}, \eqref{e-delta}
and the fact that $\Phi$ and $\Psi$ are bounded from below, it yields that there exists $M>0$ such that
\begin{align}\label{y} \frac{\gamma}{2}&\|x(T)-z\|^2+\langle \dot x(T)+\lambda\nabla\Phi(x(T))+\lambda\beta(T)\nabla\Psi(x(T)),x(T)-z\rangle\nonumber \\
& +\frac{1}{2\varepsilon}\left\|\dot x(T)+\lambda\nabla\Phi(x(T))+\lambda\beta(T)\nabla\Psi(x(T))\right\|^2\leq M \ \forall T\geq 0.
\end{align}
Combining this with \eqref{conv} and the fact that $\Phi$ and $\Psi$ are bounded from below, one can easily see that there exists $M'>0$ such that
$$\frac{d}{dT}\left(\frac{1}{2}\|x(T)-z\|^2\right)+\gamma\left(\frac{1}{2}\|x(T)-z\|^2\right)\leq M' \ \forall T\geq 0.$$
A direct application of the Gronwall Lemma implies that \begin{equation}\label{x-b}x\mbox{ is bounded} .\end{equation}
Further, this yields via \eqref{y} that
\begin{equation}\label{bound}\dot x+\lambda\nabla\Phi(x)+\lambda\beta\nabla\Psi(x)\mbox{ is bounded}.\end{equation}
From \eqref{x-b}, \eqref{bound}, \eqref{e-eps} and\eqref{e-delta} we conclude that
\begin{equation}\label{e-eps-bound}E\mbox{ is bounded from below}.\end{equation}

Moreover, from \eqref{h-e-1'}, \eqref{ineq-conj} and \eqref{phi-conv} we obtain for almost every $t \in [0,+\infty)$
\begin{align} & \dot E(t)+
\widetilde\beta(t)\left(-\Psi^*\left(\frac{-\nabla\Phi(z)}{\widetilde\beta(t)}\right)+\sigma_{\argmin \Psi}\left(\frac{-\nabla\Phi(z)}{\widetilde\beta(t)}\right)\right)\nonumber\\
\leq & \ \dot E(t) +\widetilde\beta(t)\Psi(x(t))+\langle \nabla\Phi(z),x(t)-z\rangle\nonumber\\
\leq &\ \dot E(t)
+\Phi(x(t))-\Phi(z)+\widetilde\beta(t)\Psi(x(t))\nonumber\\\leq & \ 0.\label{ineq2'}
\end{align}

(i) Consider the function $F:[0,+\infty)\rightarrow\R$ defined by
$$F(t)=\int_0^t\left(-\widetilde\beta(s)\Psi(x(s))-\langle\nabla\Phi(z),x(s)-z\rangle\right)ds \ \forall t\geq 0.$$

Making again use of (see \eqref{ineq2'})
$$-\widetilde\beta(s)\Psi(x(s))-\langle\nabla\Phi(z),x(s)-z\rangle\geq \dot E(s) \ \forall s \in [0,+\infty)$$ and \eqref{e-eps-bound}, we easily derive that
$F$ is bounded from below. Moreover, from \eqref{ineq-conj} it follows that for almost every $t \in [0,+\infty)$
$$\dot F(t)\leq \widetilde\beta(t)\left(\Psi^*\left(\frac{-\nabla\Phi(z)}{\widetilde\beta(t)}\right)
-\sigma_{\argmin \Psi}\left(\frac{-\nabla\Phi(z)}{\widetilde\beta(t)}\right)\right).$$
Notice that according to (H), the function on the right-hand side of this inequality is $L^1$-integrable on $[0,+\infty)$,
hence a direct application of
Lemma \ref{fejer-cont1} yields that $\lim_{t\rightarrow+\infty}F(t)$ exists and is a real number. Thus
\begin{equation}\label{b}\exists \lim_{t\rightarrow+\infty}\int_0^t\left(\widetilde\beta(s)\Psi(s)+\langle \nabla \Phi(z),x(s)-z\rangle\right)ds\in\R.\end{equation}

Since $\Psi\geq 0$, we obtain for every $t \in [0,+\infty)$
$$\widetilde\beta(t)\Psi(x(t))+\langle \nabla \Phi(z),x(t)-z\rangle\geq \frac{\widetilde\beta(t)}{2}\Psi(x(t))+\langle \nabla \Phi(z),x(t)-z\rangle$$ and from here,
similarly to \eqref{ineq-conj},
$$-\frac{\widetilde\beta(t)}{2}\Psi(x(t))-\langle \nabla \Phi(z),x(t)-z\rangle\leq
\frac{\widetilde\beta(t)}{2}\left(\Psi^*\left(\frac{-2\nabla \Phi(z)}{\widetilde\beta(t)}\right)-\sigma_{\argmin \Psi}\left(\frac{-2\nabla \Phi(z)}{\widetilde\beta(t)}\right)\right).$$

Thus, for almost every $t \in [0,+\infty)$ it holds
\begin{align*} & \dot E(t)+
\frac{\widetilde\beta(t)}{2}\left(-\Psi^*\left(\frac{-2\nabla \Phi(z)}{\widetilde\beta(t)}\right)+\sigma_{\argmin \Psi}\left(\frac{-2\nabla\Phi(z)}{\widetilde\beta(t)}\right)\right)\\
\leq & \ \dot E(t)+\frac{\widetilde\beta(t)}{2}\Psi(x(t))+\langle \nabla \Phi(z),x(t)-z\rangle\\
\leq & \ \dot E(t) +\widetilde\beta(t)\Psi(x(t))+\langle \nabla \Phi(z),x(t)-z\rangle\\
\leq & \ 0.
\end{align*}
Following the same technique as in the proof of \eqref{b}, it yields that
\begin{equation}\label{b2}\exists \lim_{t\rightarrow+\infty}\int_0^t\left(\frac{\widetilde\beta(s)}{2}\Psi(s)+\langle \nabla \Phi(z),x(s)-z\rangle\right)ds\in\R.\end{equation}
Finally, from \eqref{b}, \eqref{b2} and \eqref{def-beta-tilde} we obtain (i).

(ii) Follows from (i), \eqref{b} and \eqref{def-beta-tilde}.

(iii) Follows from \eqref{ineq2'} and \eqref{ineq-conj}, by following the same arguments as for proving statement \eqref{b}.

(iv) Follows by integrating \eqref{ineq3}, by taking into account $(H)$, \eqref{e-eps-bound}, \eqref{eps} and the fact that $\theta\in(0,1)$.
\end{proof}

\begin{remark} The assumption $\lim_{t\rightarrow+\infty}\beta(t)=+\infty$ has not bee used in the above
proof. However, it will play an important role in the arguments used below.
\end{remark}

For the asymptotic analysis of the trajectories generated by the dynamical system \eqref{dyn-syst}, the continuous version of the Opial Lemma that we state as follows will be crucial.

\begin{lemma}\label{opial-cont} Let $S_\infty$ be a nonempty subset of the real Hilbert space ${\cal H}$ and $x:[0,+\infty)\rightarrow{\cal H}$
a given function. Assume that
\begin{enumerate}
\item [(i)] $\lim_{t\rightarrow+\infty}\|x(t)-z\|$ exists for every $z \in S_\infty$;
\item[(ii)] every weak limit point of $x$ belongs to $S_\infty$.
\end{enumerate}
Then there exists $x_{\infty}\in S_\infty$ such that $x(t)$ converges weakly to $x_{\infty}$ as $t\rightarrow+\infty$.
\end{lemma}

We state now the main theorem of the paper.

\begin{theorem}\label{th-nonerg-conv} Assume that $(H_\Psi)$, $(H_\Phi)$,
$(H_{\beta})$ and $(H)$ hold and let
$x:[0,+\infty)\rightarrow{\cal H}$ be the trajectory generated by the dynamical system \eqref{dyn-syst}.
Then the following statements are true:
\begin{enumerate}
 \item [(i)] $\Phi(x(t))$ converges to the optimal objective value of \eqref{opt-intr} as $t\rightarrow+\infty$;
 \item[(ii)] $\lim_{t\rightarrow+\infty}\beta(t)\Psi(x(t))=\lim_{t\rightarrow+\infty}\Psi(x(t))=0$;
 \item[(iii)] $\int_0^{+\infty}\beta(t)\Psi(x(t))dt<+\infty$;
 \item[(iv)] $\dot x$, $\nabla\Phi (x)+\beta\nabla\Psi (x)\in L^2([0,+\infty);{\cal H})$;
 \item [(v)] $\lim_{t\rightarrow+\infty}\big(\dot x(t)+\lambda\nabla\Phi (x(t))+\lambda\beta(t)\nabla\Psi (x(t))\big)=0$;
 \item[(vi)] there exists $x_{\infty}\in S$ such that $x(t)$ converges weakly to $x_{\infty}$ as $t\rightarrow+\infty$.
\end{enumerate}
\end{theorem}

\begin{proof} Fix an arbitrary $z\in S$ and consider the energy functional defined in \eqref{e-delta} for $$\delta_1=(1+\sqrt{\lambda\gamma})^2=1+\lambda\gamma+2\sqrt{\lambda\gamma}.$$
A simple computation (see \eqref{dot-e-delta}) shows that
$$\dot E_{\delta_1}(t)=(1+\sqrt{\lambda\gamma})^2\dot\beta(t)\Psi(x(t))-\left\|\sqrt{\gamma}\dot x(t)-\sqrt{\lambda}(\nabla\Phi (x(t))+\beta(t)\nabla\Psi (x(t))\right\|^2$$
Taking into account the growth condition on $\beta$, Lemma \ref{l1-op}(i) and the fact that $E_{\delta_1}$ is bounded from below, we obtain
from Lemma \ref{fejer-cont1} that
\begin{equation}\label{e1}\exists\lim_{t\rightarrow+\infty}E_{\delta_1}(t)\in\R.\end{equation}
Similarly, consider $$\delta_2=(1-\sqrt{\lambda\gamma})^2=1+\lambda\gamma-2\sqrt{\lambda\gamma}.$$
We have
$$\dot E_{\delta_2}(t)=(1-\sqrt{\lambda\gamma})^2\dot\beta(t)\Psi(x(t))-\left\|\sqrt{\gamma}\dot x(t)+\sqrt{\lambda}(\nabla\Phi (x(t))+\beta(t)\nabla\Psi (x(t))\right\|^2$$
and \begin{equation}\label{e2}\exists\lim_{t\rightarrow+\infty}E_{\delta_2}(t)\in\R.\end{equation}
From \eqref{e1} and \eqref{e2} we get $$\exists\lim_{t\rightarrow+\infty}\big(E_{\delta_1}(t)-E_{\delta_2}(t)\big)\in\R.$$
This implies by the definition of the energy functional that
\begin{equation}\label{lim-phi-psi}\exists\lim_{t\rightarrow+\infty}\big(\Phi(x(t))+\beta(t)\Psi(x(t))\big)\in\R.\end{equation}
From here we deduce \begin{equation}\label{psi-0}\lim_{t\rightarrow+\infty}\Psi(x(t))=0.\end{equation}
Indeed, since $\Phi$ is bounded from below and, by taking into account the inequality
$$\Phi(x(t))\leq \Phi(z)+\langle \nabla\Phi(x(t)),x(t)-z\rangle,$$
the fact that $x$ is bounded and $\lim_{t\rightarrow+\infty}\beta(t)=+\infty$, we have
$$\lim_{t\rightarrow+\infty}\frac{\Phi(x(t))}{\beta(t)}=0.$$ The statement \eqref{psi-0} follows now from the last relation and \eqref{lim-phi-psi}.

In the following we appeal to \cite[Lemma 3.4]{att-cza-10} for the inequality
\begin{equation}\label{liminf-phi}\liminf_{t\rightarrow+\infty}\Phi(x(t))\geq\Phi(z).\end{equation} Indeed,
this was obtained in \cite[Lemma 3.4]{att-cza-10} for a first order dynamical system.
However, the statement remains true for \eqref{dyn-syst} too, since a careful look at the proof of Lemma 3.4 in \cite{att-cza-10} reveals
that the main arguments used to obtain this conclusion are the statements in Lemma \ref{l1-op}(ii), the fact that the trajectory $x$ is bounded,
the weak lower semicontinuity of $\Psi$, equation \eqref{psi-0}, the inequality \eqref{phi-conv} and relation \eqref{opt-cond}.

From \eqref{liminf-phi} and \eqref{lim-phi-psi} we have $\lim_{t\rightarrow+\infty}\big(\Phi(x(t))+\beta(t)\Psi(x(t))\big)\geq \Phi(z)$.
We claim that \begin{equation}\label{lim-e}\lim_{t\rightarrow+\infty}\big(\Phi(x(t))+\beta(t)\Psi(x(t))\big)= \Phi(z).\end{equation}

Let us assume that $\lim_{t\rightarrow+\infty}\big(\Phi(x(t))+\beta(t)\Psi(x(t))\big)> \Phi(z)$. Then there exist $\eta>0$ and $t_0\geq 0$ such that
for every $t\geq t_0$ we have
\begin{equation}\Phi(x(t))+\beta(t)\Psi(x(t))>\Phi(z)+\eta.\end{equation}
Hence, for every $t\geq t_0$
\begin{equation}\eta<\Phi(x(t))-\Phi(z)+\left(1-\frac{k(1+\gamma\lambda)}{\varepsilon}\right)\beta(t)\Psi(x(t))+
\frac{k(1+\gamma\lambda)}{\varepsilon}\beta(t)\Psi(x(t)).\end{equation}
Integrating the last inequality and taking into account Lemma \ref{l1-op}(i) and (iii) we obtain a contradiction. In conclusion, \eqref{lim-e} holds.

Since $\Phi(x(t))\leq \Phi(x(t))+\beta(t)\Psi(x(t))$ for every $t \in [0,+\infty)$, according to \eqref{lim-e} it follows that $\limsup_{t\rightarrow+\infty}\Phi(x(t))\leq\Phi(z)$,
which combined with \eqref{liminf-phi} imply (i).

(ii) Follows from \eqref{psi-0}, \eqref{lim-e} and the fact that $\lim_{t\rightarrow+\infty}\Phi(x(t))=\Phi(z)$.

(iii)-(iv) The statements have been proved in Lemma \ref{l1-op}.

(v) From \eqref{dot-e-delta}, Lemma \ref{l1-op}(i) and Lemma \ref{fejer-cont1}, we derive that
\begin{equation}\label{exists-lim-e-lg}\exists\lim_{t\rightarrow+\infty}E_{1+\lambda\gamma}(t)\in\R.\end{equation}
Combining this with \eqref{e-delta} and \eqref{lim-phi-psi}, we obtain that
$$\exists\lim_{t\rightarrow+\infty}\left\|\dot x(t)+\lambda\nabla\Phi (x(t))+\lambda\beta(t)\nabla\Psi (x(t))\right\|\in\R.$$
The statement follows now from (iv).

(vi) This will be a consequence of the Opial Lemma.

Let us check the first statement in Lemma \ref{opial-cont}. From \eqref{l1}, $(H)$, \eqref{e-eps-bound} and Lemma \ref{fejer-cont1} we obtain
\begin{equation}\label{exists-lim-e}\exists\lim_{t\rightarrow+\infty}E(t)\in\R.\end{equation}
Further, by using \eqref{e-eps}, \eqref{exists-lim-e}, \eqref{exists-lim-e-lg}, \eqref{x-b} and (v) we conclude that
\begin{equation}\label{op1}\exists\lim_{t\rightarrow+\infty}\|x(t)-z\|\in\R.\end{equation} Since $z\in S$ was arbitrary chosen, the first
statement of the Opial Lemma is true.

We prove now that the second condition in Lemma \ref{opial-cont} is fulfilled, too.
Let $(t_n)_{n\in\N}$ be a sequence of positive numbers such that $\lim_{n\rightarrow+\infty}t_n=+\infty$ and $x(t_n)$ converges weakly
to $x_{\infty}$ as $n\rightarrow+\infty$. By using the weak lower semicontinuity of
$\Psi$ and (ii) we obtain
$$0\leq \Psi(x_{\infty})\leq\liminf_{n\rightarrow+\infty}\Psi(x(t_n))=0,$$
hence $x_{\infty}\in \argmin\Psi$. Moreover, the weak lower semicontinuity of
$\Phi$ and (i) yield
$$\Phi(x_{\infty})\leq\liminf_{n\rightarrow+\infty}\Phi(x(t_n))=\Phi(z),$$
thus $x_{\infty}\in S$.
Therefore, both conditions of Lemma \ref{opial-cont} are fulfilled and the conclusion follows.
\end{proof}

Finally, we consider the situation when the objective function of \eqref{opt-intr} is strongly convex. In this case,
the trajectory generated by \eqref{dyn-syst} converges strongly to the unique optimal solution of \eqref{opt-intr}.

\begin{theorem}\label{str-mon} Assume that $(H_\Psi)$, $(H_\Phi)$,
$(H_{\beta})$ and $(H)$ hold and let $x:[0,+\infty)\rightarrow{\cal H}$ be the trajectory generated by the dynamical system \eqref{dyn-syst}.
If $\Phi$ is strongly convex, then $x(t)$ converges strongly to the unique optimal solution of \eqref{opt-intr} as $t\rightarrow+\infty$.
\end{theorem}

\begin{proof} Let $\mu>0$ be such that $\Phi$ is $\mu$-strongly
convex. In this case the optimization problem
\eqref{opt-intr} has a unique optimal solution, which we denote by $z$.

We replace \eqref{phi-conv} with the stronger inequality
\begin{equation}\label{phi-str-conv}\frac{\mu}{2}\|x(t)-z\|^2+\Phi(z)-\Phi(x(t))\leq \langle -\nabla \Phi(z),x(t)-z\rangle \ \forall t \in [0,+\infty)\end{equation}
and obtain (see the proof of Lemma \ref{l1-op} and \eqref{ineq2'}) for almost every $t \in [0,+\infty)$
\begin{align} & \dot E(t)+\frac{\mu}{2}\|x(t)-z\|^2+
\widetilde\beta(t)\left(-\Psi^*\left(\frac{-\nabla\Phi(z)}{\widetilde\beta(t)}\right)+\sigma_{\argmin \Psi}\left(\frac{-\nabla\Phi(z)}{\widetilde\beta(t)}\right)\right)\nonumber\\
\leq & \ \dot E(t)+\frac{\mu}{2}\|x(t)-z\|^2+ +\widetilde\beta(t)\Psi(x(t))+\langle \nabla\Phi(z),x(t)-z\rangle\nonumber\\
\leq &\ \dot E(t)
+\Phi(x(t))-\Phi(z)+\widetilde\beta(t)\Psi(x(t))\nonumber\\\leq & \ 0.\label{ineq2''}
\end{align}

Taking into account $(H)$ and that $E$ is bounded from below (see \eqref{e-eps-bound}), by integration of the above inequality
we obtain that there exists a constant $C>0$ such that $$\frac{\mu}{2}\int_0^T\|x(t)-z\|^2dt\leq C \ \forall T  \geq 0.$$
According to \eqref{op1}, $\lim_{t\rightarrow+\infty}\|x(t)-z\|$ exists, thus
$\|x(t)-z\|$ converges to $0$ as $t\rightarrow+\infty$ and the proof is complete.
\end{proof}

\end{document}